\documentclass[a4paper, twoside,12pt]{article}
\usepackage{fancyhdr}
\usepackage{fnpos}
\usepackage[english]{babel}        
\usepackage[T1]{fontenc}          
\usepackage{graphicx}             
\usepackage{makeidx}
\usepackage{fancybox}
\usepackage{framed}
\usepackage{fancyhdr}
\usepackage{pstricks,pst-plot,pstricks-add}
\usepackage[margin=1in]{geometry}
\usepackage{graphicx}
\usepackage{titlesec}
\usepackage{amsmath}
\usepackage{amsfonts}   
\usepackage{amssymb}    
\usepackage{amsthm}
\usepackage{dsfont}
\usepackage{mathtools}
\usepackage{verbatim}   
\usepackage{color}
\usepackage{listings}
\usepackage{graphicx}
\usepackage{amsmath}
\usepackage{amsmath,amssymb,color,inputenc,euscript,graphicx,psfrag}

\newtheorem{theorem}{Theorem}[section]
\newtheorem{corollary}[theorem]{Corollary}
\newtheorem{lemma}[theorem]{Lemma}
\newtheorem{proposition}[theorem]{Proposition}
\newtheorem{definition}[theorem]{Definition}

\newtheorem{properties}[theorem]{Properties}

\numberwithin{equation}{section}

\DeclareMathOperator*{\supess}{sup\,ess}

\usepackage{amsmath, amsthm, amscd, amsfonts, amssymb, graphicx, color,   mathrsfs}
\usepackage[english]{babel}
\usepackage[bookmarksnumbered, colorlinks, plainpages]{hyperref}
\hypersetup{colorlinks=true,linkcolor=blue, anchorcolor=blue, citecolor=blue, urlcolor=red, filecolor=magenta, pdftoolbar=true}

\everymath{\displaystyle}
\usepackage{orcidlink} 

\usepackage{bm}

 {\markboth{\sl A. Saoudi }{\sl  Quadratic-Phase Dunkl Transform}
	
	\author {Ahmed Saoudi $^{ \orcidlink{0000-0003-4422-2054}}$}}
	
	\title{Quadratic-Phase Dunkl Transform: Fundamental properties, translation operators, convolution product  and HUP}
	
	\date{}

\begin{document}
		\maketitle
		\begin{center}
	Department of Mathematics, College of Science, Northern Border University, Arar, Saudi Arabia\\
			Email: ahmed.saoudi@ipeim.rnu.tn\\
			
		\end{center}
			\begin{abstract}
					
In this paper, we introduce and study the quadratic-phase Dunkl transform, a novel integral transform on the real line parameterized by five real numbers \((a, b, c, d, e)\) and a multiplicity parameter \(\mu \geq -1/2\). We define the transform and establish its fundamental properties, including continuity, a Riemann–Lebesgue lemma, linearity, scaling, and, most importantly, a reversibility theorem and an associated Parseval/Plancherel formula. We show that this novel quadratic-phase integral type transform generalizes a wide class of known transforms, such as  the quadratic-phase Fourier-Bessel transform, the quadratic-phase Fourier transform, the linear canonical Dunkl transform, the fractional Dunkl transform, and the classical Dunkl transform, by choosing the appropriate specialization of its parameters. Furthermore, we introduce and investigate a corresponding quadratic-phase Dunkl translation operator and a convolution structure, proving their basic properties and a Young's inequality. Finally, we establish a new Heisenberg-type uncertainty principle for the quadratic-phase Dunkl transform, which extends the classical uncertainty principle for the a large class of integral type transforms.
				\end{abstract}

\section{Introduction}
The Fourier transform played a central role in harmonic analysis, signal processing, and mathematical physics, providing an indispensable tool for analyzing functions in the frequency domain. However, the increasing complexity of modern applications, particularly in optics, quantum mechanics, and the analysis of non-stationary signals, has necessitated the development of more flexible transforms that can adapt to various physical and mathematical contexts. This evolutionary path has led to numerous generalizations of the classical Fourier transform, each designed to address specific limitations or to incorporate additional degrees of freedom.

Among these generalizations, the quadratic-phase Fourier transform has emerged as a particularly powerful and versatile extension. Introduced by Castro et al. in 2018 \cite{castro2014quadratic, castro2018new}, the quadratic-phase Fourier transform incorporates five real parameters that provide enhanced control over the transform's phase structure, enabling superior analysis of chirp-type signals and systems with quadratic phase variations. The quadratic-phase Fourier transform of a function $f \in L^1(\mathbb{R})$ is defined as:

\begin{equation*}
	\mathcal{Q}^{a,b,c}_{d,e}[f](w) = \frac{1}{\sqrt{2\pi}} \int_{\mathbb{R}} \mathcal{K}^{a,b,c}_{d,e}(w,v) f(v)  dv,
\end{equation*}

where the kernel is given by:

\begin{equation}\label{kernelQPFT}
	\mathcal{K}^{a,b,c}_{d,e}(w,v) = \exp\left\{-i(av^2 + cw^2 + bwv + dv + ew)\right\}.
\end{equation}

The parameters $a,b,c,d,e \in \mathbb{R}$ with $b \neq 0$ offer remarkable flexibility, allowing the quadratic-phase Fourier transform to encompass numerous classical transforms as special cases:

\begin{itemize}
	\item When $a = c = d = e = 0$ and $b = 1$, we recover the classical Fourier transform.
	
	\item 	When $d = e = 0$, $a = c = -\tfrac{\cot \theta}{2}$, $b = \csc \theta$ (where $\theta \neq n\pi$, $n \in \mathbb{Z}$), then  the quadratic-phase Fourier transform (amplifying by  $\sqrt{1-i \cot \theta}$) is reduced to  the fractional Fourier transform \cite{mcbride1987namias}, 
	
\begin{equation*}
		\mathscr{F}^\alpha[f](\omega)=\int_{\mathbb{R}} f(v) \mathcal{K}_\alpha(w, v) d v,
\end{equation*}
 where the kernel  $\mathcal{K}_\alpha(w, v)$   is given by 
 
	\begin{equation}\label{defFrFT}
		\mathcal{K}_\theta(w, v) = \begin{cases}
			\sqrt{\tfrac{1-i \cot \theta}{2\pi}} \exp\left\{\tfrac{i}{2}(w^2+v^2) \cot \theta - i\omega v \csc \theta\right\}, & \theta \neq k\pi \\
			\delta(v-w), & \theta = 2k\pi \\
			\delta(v+w), & \theta = (2k+1)\pi.
		\end{cases}
	\end{equation}
	
	\item  If $d=e=0$, $b\neq 0$, and under  the transformations $a \mapsto - a/2b$, $b \mapsto 1/b$ and $c \mapsto - c/2b$, then   the quadratic-phase Fourier transform (amplifying by  $\tfrac{1}{\sqrt{ib}}$) is reduced to the linear canonical transform \cite{recentdevLCT2007}
	
	\begin{equation}
		\mathcal{L}[f](w) = \frac{1}{\sqrt{2\pi i b}} \int_{\mathbb{R}} f(v) \exp\left\{\frac{i(av^2 - 2ww +  cw^2)}{2b}\right\}  dv.
	\end{equation}
\end{itemize}

The quadratic-phase Fourier transform  is a recent addition to the class of integral transforms, encompassing various signal processing tools such as the Fourier transform, fractional Fourier transform and linear canonical transform \cite{shah2021short, shah2021uncertainty, shah2023sampling}. 
The quadratic-phase Fourier transform is particularly noted for its additional degrees of freedom, which enhance its performance in time-frequency analysis applications \cite{zayed2025discrete}.

The quadratic phase Fourier transform was developed and studied in various setting, including wavelet transform \cite{prasad2020quadratic, shah2022quadratic, sharma2023quadratic}, uncertainty principles \cite{castro2023uncertainty, castro2025multidimensional, shah2021uncertainty},  quaternion domains 
\cite{ahmad2025short, bhat2024novel, gargouri2024novel, gupta2024short,  dar2207towards,  zayed2025discrete}, octonion domains \cite{kumar2024octonion, lone2025analysis}, short-time quadratic-phase setting \cite{ahmad2025short, gupta2024short, shah2021short}, Stockwell domains \cite{dar2024convolution, gargouri2024novel}, Bessel setting \cite{QFFBT}, Fourier--Jacobi framework \cite{QPJFT}, Jacobi Dunkl setting \cite{QPJDT} and Hankel setting \cite{prasad2020convolution, roy2025pseudo, roy2025wavelet}.

Despite these advances, the quadratic-phase Fourier transform operates within the conventional Euclidean setting and does not account for systems with underlying symmetries or weight functions. The limitation arising from this motivates our present work, introducing the quadratic-phase Dunkl transform, a novel generalization synergistically fusing the quadratic-phase Fourier transform's parametric phase structure together with the rich theoretical framework of Dunkl operators associated with reflection groups. This unification gives one a strong analytical tool capable of handling both quadratic phase variations and symmetric structures in weighted function spaces.

The Dunkl transform, a generalization of the Fourier transform to spaces invariant under the action of finite reflection groups, has recently found extensive applications in multivariate special functions, quantum physics, and studies on differential-difference equations. Including Dunkl operators in the quadratic-phase Fourier transform framework gives us a transform that encompasses both the classical quadratic-phase Fourier transform and the Dunkl transform as special cases but also offers new possibilities for the analysis of signals and systems with reflection symmetries in weighted spaces.

In the present article, we introduce and study the  quadratic-phase Dunkl  transform that depend on five real parameters \( a, b, c, d, e \) and multiplicity function $\mu\geq -1/2$, where $b\neq 0$. The quadratic-phase Dunkl transform of an integrable function $f$ on the real line as

\begin{equation*}
	\mathcal{D}^{a,b,c}_{d,e,\mu}[f](w)= \frac{c_\mu}{( i b)^{\mu+1}} \int_{\mathbb{R}} \Psi^{a,b,c}_{d,e,\mu}(w,v) f(v)|v|^{2 \mu+1} d v,
\end{equation*}
where $\Psi^{a,b,c}_{d,e,\mu}(w,v)$ denotes the quadratic-phase Dunkl kernel, given by

\begin{equation*}
	\Psi^{a,b,c}_{d,e,\mu}(w,v)= e^{-i(av^2+cw^2+dv+ew)} E_{\mu}(-i w / b, v),
\end{equation*}
with $E_{\mu}(w,v)$ is the Dunkl kernel  given by the relation (\ref{dkernel}).

The quadratic-phase Dunkl transform serves as a unifying framework for a wide class of integral transforms in harmonic analysis. By making appropriate choices of the parameters \(a, b, c, d, e\) and the multiplicity parameter \(\mu\), it reduces to numerous well-known transforms. For instance, it encompasses the quadratic-phase Fourier-Bessel transform, the linear canonical Dunkl transform, and the fractional Dunkl transform. Moreover, in the case \(\mu = -1/2\), it recovers the classical quadratic-phase Fourier transform and its special cases such as the linear canonical transform, the fractional Fourier transform, and the standard Fourier transform. Additionally, it includes the Dunkl transform and the Dunkl-type Fresnel transform as particular instances. This versatility demonstrates the transformative power of the quadratic-phase Dunkl transform in bridging various generalizations of the Fourier transform.

In this paper, we develop the comprehensive theory of the quadratic-phase Dunkl transform, establishing its fundamental properties including boundedness, continuity, and a direct expression in terms of the classical Dunkl transform. We also established a Riemann–Lebesgue result, an inversion formula, and Parseval/Plancherel identities. These results confirm that the quadratic-phase Dunkl transform acts as a unitary operator on \(L^2_\mu(\mathbb{R})\).
Furthermore, we introduced the translation and convolution operators linked to this transform and verified their main structural properties together with a Young-type inequality. We then obtained a Heisenberg-type uncertainty principle for the quadratic-phase Dunkl transform, which extends the classical uncertainty principle to this more general framework.

 The outline of this paper is organized as follows. In Section \ref{sec2}, we recall some necessary preliminaries on the Dunkl transform and related special functions. Section \ref{sec3} introduces the quadratic-phase Dunkl transform, examines its fundamental properties, and details its special cases. In Section \ref{sec4}, we define and study the associated quadratic-phase Dunkl translation operator. Section \ref{sec5}, is devoted to the corresponding convolution product and its main properties. Finally, Section \ref{sec6}, establishes a Heisenberg-type uncertainty principle for the quadratic-phase Dunkl transform.

\section{Preliminaries}\label{sec2}

Along this article, we consider the following notation and function spaces.
\begin{itemize}
	
	\item $\mathcal{C}_{0}(\mathbb{R})$ :  represents the space of continuous functions on $ \mathbb{R} $ that vanish at infinity, equipped with the topology of uniform convergence.
	
	\item  $L^{p}_{\mu}(\mathbb{R})$ denotes the Lebesgue space of measurable functions $f$ on $\mathbb{R}$ such that
	$$
	\begin{array}{l}
		\|f\|_{\mu,p}=\left(\int_{\mathbb{R}}|f(v)|^{p} |v|^{2\mu+1} dv\right)^{\frac{1}{p}}<\infty, \quad \text { if } 1 \leq p<\infty \\
		\|f\|_{\mu,\infty}=\displaystyle\supess_{v\in\mathbb{R}}|f(v)|<\infty, \quad \text { if } p=\infty,
	\end{array}
	$$
	provided with the topology defined by the norm $\|\cdot\|_{\mu ,p}$, with $\mu\geq -1/2$.
	\item $\mathcal{S}\left(\mathbb{R}\right)$  denotes the Schwartz space of functions on \( \mathbb{R} \) that decay rapidly.
\end{itemize}

For $\mu \geq  -1/2$, the differential-difference operators $\Lambda_{\mu}$ related to the reflection group $\mathbb{Z}_{2}$ on $\mathbb{R}$ is given by

$$
\Lambda_{\mu} f(v)=\frac{\mathrm{d}}{\mathrm{~d} v} f(v)+\frac{2 \mu+1}{2 v}[f(v)-f(v)],
$$
called the Dunkl operator on $\mathbb{R}$ of index $\mu+1/2$ associated with the reflection group $\mathbb{Z}_{2}$, see [19,20].

For $\mu \geq- 1/2$ and $w \in \mathbb{C}$, the initial value problem

\[
\begin{cases}
	\Lambda_{\mu} f(v)=w f(v), \quad v \in \mathbb{R}, \\ 
	f(0)=1,
\end{cases}
\]
admits a unique solution denoted by $E_{\mu}(w,v)$ and called the  kernel of classical Dunkl operator given by (see \cite{Rosler2003})
\begin{equation}\label{dkernel}
	E_{\mu}(i w, v)=j_{\mu}(wv)+\frac{i wv}{2(\mu+1)} j_{\mu+1}(wv),
\end{equation}
where $j_{\mu}$ represent the normalized spherical Bessel function
\begin{equation}\label{seriesb}
	j_{\mu}(w):=2^{\mu} \Gamma(\mu+1) \frac{J_{\mu}(w)}{w^{\mu}}=\Gamma(\mu+1) \sum_{n=0}^{+\infty} \frac{(-1)^{n}(w / 2)^{2 n}}{n ! \Gamma(n+\mu+1)},
\end{equation}
and the constant $	c_\mu$ is given by
\begin{equation}
	c_\mu= \frac{1}{2^{\mu+1}\Gamma(\mu+1)}.
\end{equation}
Here $J_{\mu}$ denotes the classical Bessel function (see \cite{Watson1944}).

Note that this kernel  has the following important properties.\\
For $w, v \in \mathbb{C}$, we have $E_{\mu}(w,v)=E_{\mu}(v,w) ; E_{\mu}(w,0)=1$ and

\begin{equation*}
	\left|E_{\mu}(i w, v)\right| \leq 1.
\end{equation*}
For $f \in L^{1}_{\mu}(\mathbb{R})$, the Dunkl transform $\mathcal{D}_\mu$ is given by

\begin{equation*}
 \mathcal{D}_\mu[f](w)= c_\mu \int_{\mathbb{R}} E_{\mu}(-i w, v) f(v)|v|^{2 \mu+1} d v.
\end{equation*}

In the following, we state some fundamental properties for the Dunkl transform.
\begin{theorem}
\begin{enumerate}
\item  For all $f \in L^{1}_{\mu}(\mathbb{R})$ such that $\mathcal{D}_\mu f\ \in L^{1}_{\mu}(\mathbb{R})$, the inversion formula for the Dunkl transform is given by

\begin{equation*}
	f(v)= c_\mu \int_{\mathbb{R}} E_{\mu}(i w, v) f(w)|w|^{2 \mu+1} d w.
\end{equation*}

\item  For all $f, g \in L^{2}_{\mu}(\mathbb{R})$,  the  Parseval's formula for  the Dunkl transform is given by

\begin{equation*}
	\int_{\mathbb{R}} f(v) \overline{g(v)} |v|^{2 \mu+1} d v = \int_{\mathbb{R}}\mathcal{D}_\mu [f](w) \overline{\mathcal{D}_\mu [g](w)} |w|^{2 \mu+1} d w.
\end{equation*}

\item  For all $f \in L^{2}_{\mu}(\mathbb{R})$  the Plancherel's formula for  the Dunkl transform is given by

\begin{equation*}
	\int_{\mathbb{R}}|f(v)|^{2} |v|^{2 \mu+1} d v = \int_{\mathbb{R}}\left|\mathcal{D}_\mu [f](w)\right|^{2} |w|^{2 \mu+1} d w.
\end{equation*}
\end{enumerate}
\end{theorem}

\section{Fundamental properties of  quadratic-phase Dunkl  transform}\label{sec3}

We began this section by introducing the definition of the quadratic-phase Dunkl transform and stating its particular cases using five real parameters (a, b, c, d, e) and the multiplicity function \(\mu \geq -1/2\).

Next we developed its fundamental properties, including the reversibility property and Parseval’s formula.

\begin{definition}
We define  the quadratic-phase Dunkl transform of an integrable function $f$ on the real line as 

\begin{equation}\label{QFDT}
	\mathcal{D}^{a,b,c}_{d,e,\mu}[f](w)= \frac{c_\mu}{( i b)^{\mu+1}} \int_{\mathbb{R}} \Psi^{a,b,c}_{d,e,\mu}(w,v) f(v)|v|^{2 \mu+1} d v,
\end{equation}
where $\Psi^{a,b,c}_{d,e,\mu}(w,v)$ denotes the quadratic-phase Dunkl kernel, given by

\begin{equation}\label{QFDTK}
	\Psi^{a,b,c}_{d,e,\mu}(w,v)= e^{-i(av^2+cw^2+dv+ew)} E_{\mu}(-i w / b, v),
\end{equation}
with $E_{\mu}(w,v)$ is the Dunkl kernel  given by the relation (\ref{dkernel}).
\end{definition}
\subsection*{Particular cases}
With suitable choices of parameters  \( a, b, c, d, e \) and multiplicity function $\mu$, the  quadratic-phase Dunkl  transform (\ref{QFDT}) reduces to several well-known integral transforms, as shown below:
\begin{itemize}
	\item The even part of  quadratic-phase Dunkl  transform reduced to the  quadratic-phase Fourier-Bessel  transform defined by  \cite{QFFBT}:
	
	\begin{equation}\label{QF.FBT}
		\mathcal{B}^{a,b,c}_{d,e,\mu}[f](w)= \frac{2c_\mu}{( i b)^{\mu+1}} \int_{0}^\infty \mathcal{J}^{a,b,c}_{d,e,\mu}(w,v) f(v)v^{2 \mu+1} d v , \quad b \neq 0,
	\end{equation}
	where $\mathcal{J}^{a,b,c}_{d,e,\mu}(w,v)$ denotes the quadratic-phase Fourier-Bessel kernel,  given by
	
	\begin{equation}\label{QF.FBTK}
		\mathcal{J}^{a,b,c}_{d,e,\mu}(w,v)= e^{-i(av^2+cw^2+dv+ew)} j_{\mu}\left(wv/b\right),
	\end{equation}
	where $j_{\mu}$ is the normalized spherical Bessel function given by the relation (\ref{seriesb}).
	
	\item If  $a=d=1$, $c=\tau$ and $c=e=0$, then, the quadratic-phase Dunkl is transform reduced to one dimension Dunkl type Fresnel transform defined by R\"{o}sler \cite{Roslerfresnel}
				
	\begin{equation*}
		\mathcal{D}_{\mu}^{\tau}(f)(w)=\left\{\begin{array}{ll}
			\displaystyle\frac{1}{\Gamma(\mu+1)(2 i \tau)^{\mu+1}} \int_{\mathbb{R}} e^{\frac{i}{2 \tau}\left(w^{2}+v^{2}\right)} E_{\mu}(-i w / \tau, v) f(v)|v|^{2 \mu+1} dv, & \tau \neq 0 \\\\
			f(x), & \tau=0.
		\end{array}\right.
	\end{equation*}
	
	\item If $d=e=0$, $b\neq 0$,  and under  the transformations $a \mapsto - a/2b$ and $c \mapsto - c/2b$, we obtain the one dimensional	 linear canonical Dunkl transform \cite{Ghazouani-unified17, saoudi2024hardy}
	
	\begin{equation}\label{LCDT}
		\mathcal{L}^\mathcal{D}_{\mu} [f](w)= \frac{c_\mu}{( i b)^{\mu+1}}\int_{\mathbb{R}} \mathcal{K}_{\mu}(w,v) f(v)|v|^{2 \mu+1} d v,\quad   b \neq 0,
	\end{equation}
	where $\mathcal{K}_{\mu}(w,v)$ represents the linear canonical Dunkl  kernel defined by 
	
	\begin{equation}\label{LCDT.K}
		\mathcal{K}_{\mu}(w,v)=e^{\frac{i}{2}\left(\frac{a}{b} v^{2}+\frac{d}{b} w^{2}\right)} E_{\mu}(i w, v),
	\end{equation}
	and  $E_{\mu}(w,v)$ is the Dunkl kernel  given by the relation (\ref{dkernel}).
	
	\item If $d=e=0$, $b\neq 0$,  and under  the transformations $a \mapsto - a/2b$ and $c \mapsto - c/2b$, then the even part of  quadratic-phase Dunkl  transform reduced to the  linear canonical Fourier-Bessel  transform defined  by  \cite{ dhaouadi2020harmonic, mohamed2024linear}
	
	\begin{equation}\label{LCFBT}
		\mathcal{L}^\mathcal{B}_{\mu} [f](w)= \frac{2c_\mu}{( i b)^{\mu+1}}\int_{\mathbb{R}} \mathscr{K}_{\mu}(w,v) f(v)|v|^{2 \mu+1} d v,\quad   b \neq 0,
	\end{equation}
	where $E_{\mu}^{M}(w,v)$ represents the linear canonical Fourier-Bessel  kernel defined by 
	
	\begin{equation}\label{LCFBT.K}
		\mathscr{K}_{\mu}(w,v)=e^{\frac{i}{2}\left(\frac{a}{b} v^{2}+\frac{d}{b} w^{2}\right)} j_{\mu}\left(wv/b\right).
	\end{equation}

	\item If $d = e = 0$, $a = c = -\tfrac{\cot \theta}{2}$, $b = \sin\theta $ (where $\theta \neq n\pi$, $n \in \mathbb{Z}$), then the   quadratic-phase Dunkl  transform reduced to  the fractional Dunkl transform (except for a constant unimodular factor $\left(e^{i \theta}\right)^{\mu+1}$ defined by Ghazouani et al.  \cite{Ghazouani14, Ghazouani16}:
	\begin{equation*}
		\mathcal{D}_{\mu}^{\theta} f(w)=\left\{\begin{array}{ll}
			\displaystyle A_{\mu,\theta}\int_{\mathbb{R}} e^{\tfrac{i}{2}\left(v^{2}+w^{2}\right) \cot \alpha} E_{\mu}\left(\tfrac{-iw}{\sin(\theta)} , v\right) f(v)|v|^{2 \mu+1} d v, & \pi\in\mathbb{R}-n\pi\mathbb{Z}\\
			f(w), & \theta=2 n \pi \\
			f(-w), & \theta=(2 n+1) \pi,
		\end{array}\right.
	\end{equation*}
	where $A_{\mu,\theta}=\tfrac{ e^{i(\mu+1)(\hat{\theta}\pi/2-(\theta-2n \pi))}}{\Gamma(\mu+1)(2|\sin \theta|)^{\mu+1}}$,  $n\in \mathbb{Z}$ and $\hat{\theta}=sgn(\sin\theta)$.
	
		\item If  $a = c = d = e = 0$ and $b = 1$, then the quadratic-phase Dunkl  transform is to reduced to the Dunkl transform  \cite{dunkl1989differential, dunkl1991integral, dunkl1992Hankel}:
	\begin{equation*}
		\mathcal{D}_{\mu} f(x)=\frac{1}{2^{\mu+1} \Gamma(\mu+1)} \int_{\mathbb{R}} f(x) E_{\mu}(-ix,y)|y|^{2 \mu+1} d y.
	\end{equation*}
	
	\item If $\mu=-1/2$, then under  the transformation $b \mapsto 1/b$ the quadratic-phase Dunkl  transform reduced to the  quadratic-phase Fourier transform defined by \cite{castro2018new} 
	
	\begin{equation}\label{QFF.T}
		\mathcal{Q}^{a,b,c}_{d,e}[f](w)= \frac{1}{ \sqrt{2\pi}} \int_{\mathbb{R}} \mathcal{K}^{a,b,c}_{d,e}(w,v) f(v)d v,
	\end{equation}
	where $\mathcal{K}^{a,b,c}_{d,e}(w,v)$ denotes the quadratic-phase Fourier kernel, given by (\ref{kernelQPFT}).
%
%
		\item  If  $\mu=-\tfrac{1}{2}$, $d=e=0$, $b\neq 0$, then, under  the transformations $a \mapsto - a/2b$ and $c \mapsto - c/2b$,  the quadratic-phase Dunkl (amplifying by  $\tfrac{1}{\sqrt{ib}}$) is reduced to the linear canonical transform  \cite{recentdevLCT2007}
	
	\begin{equation}
		\mathcal{L}[f](w) = \frac{1}{\sqrt{2\pi i b}} \int_{\mathbb{R}} f(v) \exp\left\{\frac{i(av^2 - 2ww +  cw^2)}{2b}\right\}  dv.
	\end{equation}

\item 	If $\mu=-\tfrac{1}{2}$,  $d = e = 0$, $a = c = -\tfrac{\cot \theta}{2}$, $b = \csc \theta$ (where $\theta \neq n\pi$, $n \in \mathbb{Z}$), then  the quadratic-phase Fourier transform (amplifying by  $\sqrt{1-i \cot \theta}$) is reduced to  the fractional Fourier transform, 

\begin{equation*}
	\mathscr{F}^\theta[f](\omega)=\int_{\mathbb{R}} f(v) \mathcal{K}_\theta(w, v) d v,
\end{equation*}
where the kernel  $\mathcal{K}_\theta(w, v)$   is given by (\ref{defFrFT}).

	\item If $\mu=-\tfrac{1}{2}$,    $a = c = d = e = 0$ and $b = 1$, then  the quadratic-phase Fourier transform is reduced to the classical Fourier transform given by
	$$
	\mathcal{F}(f)(\xi)=\frac{1}{\sqrt{2 \pi}} \int_{-\infty}^{+\infty} f(x) e^{-i x \xi} d x.
	$$
\end{itemize}

\begin{lemma} Let  \( a, b, c, d, e \)  five real parameters such that $b\neq 0$ and $\mu\geq -1/2$.
 Then for all  $v,w\in \mathbb{R}$, we have:
 \begin{enumerate}
 	\item The  quadratic-phase Dunkl kernel is bounded and we have
		\begin{equation}\label{ModuleQFDTK}
			\left|\Psi^{a,b,c}_{d,e,\mu}(w,v)\right|\leq 1.
		\end{equation}	
	\item 	$\overline{\Psi^{a,b,c}_{d,e,\mu}(w,v)}=\Psi^{-c,-b,-a}_{-e,-d,\mu}(w,v)$.  
	\end{enumerate}
\end{lemma}

\begin{proof} 
For \( a, b, c, d, e \in \mathbb{R}\) such that $b\neq 0$ and $\mu\geq -1/2$.
 \begin{enumerate}
	\item We have 
\begin{eqnarray*}
  \left|\Psi^{a,b,c}_{d,e,\mu}(w,v)\right|&=& \left| e^{-i(av^2+cw^2+dv+ew)} E_{\mu}(-i w / b, v) \right|\\
  &\leq&  \left| E_{\mu}(-i w / b, v) \right|.
\end{eqnarray*}
According to \cite[Corollary 5.4]{rosler1999positivity}, we deduce that

		\begin{equation*}\
	\left|\Psi^{a,b,c}_{d,e,\mu}(w,v)\right|\leq 1.
\end{equation*}
\item The proof is trivial.
\end{enumerate}
\end{proof}

\begin{theorem}\label{cont1QPDT}
Let \( a, b, c, d, e \in \mathbb{R}\) such that $b\neq 0$,  $\mu \geq  -1/2$. 
\begin{enumerate}
 \item If  $f$ be a function belongs to  $L^{1}_{\mu}(\mathbb{R})$. Then the maps $w\mapsto\mathcal{D}^{a,b,c}_{d,e,\mu}[f](w)$ is continuous on the real line and we have
 
 \begin{equation}
 	\left|\mathcal{D}^{a,b,c}_{d,e,\mu}[f](w)\right|\leq \frac{c_\mu}{ |i b|^{\mu+1}}	\|f\|_{\mu,1}.
 \end{equation}
  \item The quadratic-phase Dunkl transform admits a fundamental representation in terms of the classical Dunkl transform. For any function  $f\in \mathcal{S}\left(\mathbb{R}\right)$, the quadratic-phase Dunkl transform can be expressed as
 
\begin{equation}\label{relDTvsQPDT}
\mathcal{D}^{a,b,c}_{d,e,\mu}[f](w) = \frac{1}{( i b)^{\mu+1}} e^{-i(cw^2+ew)} \mathcal{D}_\mu(h)(w / b),
\end{equation}
 with $h(v)=e^{-i(av^2+dv)}f(v).$ 
\end{enumerate}
\end{theorem}

\begin{proof} We consider  \( a, b, c, d, e \in \mathbb{R}\) such that $b\neq 0$,  $\mu\geq -1/2$.
\begin{enumerate}
	\item Let $f$ be function in  $L^{1}_{\mu}(\mathbb{R})$. Since the maps $w\mapsto 	E_{\mu}(i w, v)$ is continuous on $\mathbb{R}$, we deduce that
	 $$w\mapsto\mathcal{D}^{a,b,c}_{d,e,\mu}[f](w)$$ 
	  is also continuous maps on the real line for every \( x \in \mathbb{R} \). Furthermore,
	\begin{eqnarray*}
         	\left|\mathcal{D}^{a,b,c}_{d,e,\mu}[f](w)\right| &=&\left| \frac{c_\mu}{( i b)^{\mu+1}} \int_{\mathbb{R}} \Psi^{a,b,c}_{d,e,\mu}(w,v) f(v)|v|^{2 \mu+1} d v \right| \\
         	&\leq & \frac{c_\mu}{ |i b|^{\mu+1}} \int_{\mathbb{R}} \left|\Psi^{a,b,c}_{d,e,\mu}(w,v) f(v)\right| |v|^{2 \mu+1} d v  \\
         		&\leq& \frac{c_\mu}{ |i b|^{\mu+1}} \int_{\mathbb{R}} \left| f(v)\right| |v|^{2 \mu+1} d v  \\
         		& \leq& \frac{c_\mu}{ |i b|^{\mu+1}}	\|f\|_{\mu,1}.
	\end{eqnarray*}
	
	\item Let   $f$ be a function in $ \mathcal{S}\left(\mathbb{R}\right)$, then we have
	\begin{eqnarray*}
		\mathcal{D}^{a,b,c}_{d,e,\mu}[f](w)&=& \frac{c_\mu}{( i b)^{\mu+1}} \int_{\mathbb{R}} \Psi^{a,b,c}_{d,e,\mu}(w,v) f(v)|v|^{2 \mu+1} d v
		\\
		 &=& \frac{c_\mu}{( i b)^{\mu+1}} \int_{\mathbb{R}} e^{-i(av^2+cw^2+dv+ew)} E_{\mu}(-i w / b, v) f(v)|v|^{2 \mu+1} d v \\
		&=& \frac{1}{( i b)^{\mu+1}} e^{-i(cw^2+ew)} \left[ c_\mu \int_{\mathbb{R}} E_{\mu}(-i w / b, v) e^{-i(av^2+dv)} f(v) |v|^{2 \mu+1} d v \right] \\
		&=& \frac{1}{( i b)^{\mu+1}} e^{-i(cw^2+ew)} \mathcal{D}_\mu(h)(w / b),
	\end{eqnarray*}
	 with $h(v)=e^{-i(av^2+dv)}f(v).$ 
\end{enumerate}  
\end{proof}

\begin{proposition}{(Riemann–Lebesgue lemma for Quadratic-Phase Dunkl Transform)}
	Let \( a, b, c, d, e \) be five real parameters such that $b\neq 0$ and $\mu \geq -1/2$. For every $f\in L^{1}_{\mu}(\mathbb{R})$, the quadratic-phase Dunkl transform $\mathcal{D}^{a,b,c}_{d,e,\mu}[f]$ belongs to \( \mathcal{C}_{0}(\mathbb{R}) \) and satisfies
	\begin{equation}\label{RLLQFDT}
		\left\|\mathcal{D}^{a,b,c}_{d,e,\mu}[f]\right\|_{\infty}\leq \frac{c_\mu}{ |b|^{\mu+1}}	\|f\|_{\mu,1}.
	\end{equation}
\end{proposition}

\begin{proof}
	Let \( f \in L^{1}_{\mu}(\mathbb{R}) \). From the first assertion of Theorem \ref{cont1QPDT}, we know that the transform $\mathcal{D}^{a,b,c}_{d,e,\mu}[f]$ is continuous on $\mathbb{R}$ and bounded by $\frac{c_\mu}{ |b|^{\mu+1}}	\|f\|_{\mu,1}$.
	
	To show that it vanishes at infinity, we first consider the case when \( f \in \mathcal{S}(\mathbb{R}) \). Using the relation (\ref{relDTvsQPDT}) between the Dunkl transform and the quadratic phase Dunkl transform established in Theorem \ref{cont1QPDT}:
	
	\[
	\mathcal{D}^{a,b,c}_{d,e,\mu}[f](w) = \frac{1}{( i b)^{\mu+1}} e^{-i(cw^2+ew)} \mathcal{D}_\mu(h)(w / b),
	\]
	where \( h(v)=e^{-i(av^2+dv)}f(v) \). Since \( f \in \mathcal{S}(\mathbb{R}) \), it follows that \( h \in \mathcal{S}(\mathbb{R}) \). By the classical Riemann-Lebesgue lemma for the Dunkl transform, we have that \( \mathcal{D}_\mu(h) \in \mathcal{C}_{0}(\mathbb{R}) \). The composition \( \mathcal{D}_\mu(h)(-i w / b) \) and multiplication by the continuous bounded function \( e^{-i(cw^2+ew)} \) preserve the property of vanishing at infinity. Hence, \( \mathcal{D}^{a,b,c}_{d,e,\mu}[f] \in \mathcal{C}_{0}(\mathbb{R}) \) for \( f \in \mathcal{S}(\mathbb{R}) \).
	
	Now, for a general \( f \in L^{1}_{\mu}(\mathbb{R}) \), let \( \{f_n\} \subset \mathcal{S}(\mathbb{R}) \) be a sequence such that \( \|f_n - f\|_{\mu,1} \to 0 \) as \( n \to \infty \). From the uniform estimate \eqref{RLLQFDT}, we have:
	
	\[
	\left\| \mathcal{D}^{a,b,c}_{d,e,\mu}[f_n] - \mathcal{D}^{a,b,c}_{d,e,\mu}[f] \right\|_{\infty} \leq \frac{c_\mu}{ |b|^{\mu+1}} \|f_n - f\|_{\mu,1} \to 0.
	\]
	Since the uniform limit of functions in \( \mathcal{C}_{0}(\mathbb{R}) \) also belongs to \( \mathcal{C}_{0}(\mathbb{R}) \), we conclude that \( \mathcal{D}^{a,b,c}_{d,e,\mu}[f] \in \mathcal{C}_{0}(\mathbb{R}) \), completing the proof.
\end{proof}

\begin{proposition} 
	Let \( a, b, c, d, e \in \mathbb{R}\) such that $b\neq 0$, $\mu \geq -1/2$, $\alpha,\beta\in\mathbb{R}$ and $k\in\mathbb{R}_+$. 
	For every $f,g\in L^{1}_{\mu}(\mathbb{R})$, the quadratic-phase Dunkl transform satisfies the following properties:
	\begin{enumerate}
		\item \textbf{Linearity}: $\mathcal{D}^{a,b,c}_{d,e,\mu}[\alpha f+\beta g](w) = \alpha\mathcal{D}^{a,b,c}_{d,e,\mu}[f](w) + \beta\mathcal{D}^{a,b,c}_{d,e,\mu}[g](w)$.
		
		\item \textbf{Scaling}: $ \mathcal{D}^{a,b,c}_{d,e,\mu}[f](k w)=  \frac{1}{k^{2\mu+2}} \mathcal{D}^{a',b,c'}_{d',e',\mu} [f_k](w) $, 
		with $ a' = \frac{a}{k^2}$, $c' = c k^2$, $d' = \frac{d}{k}$, $e' = e k$ and $ f_k(v) = f(v/k)$.
	\end{enumerate}
\end{proposition}

\begin{proof} 
	Let $f,g\in L^{1}_{\mu}(\mathbb{R})$, \( a, b, c, d, e \in \mathbb{R}\) such that $b\neq 0$, $\mu \geq -1/2$, $\alpha,\beta\in\mathbb{R}$ and $k\in\mathbb{R}_+$.
	
	\begin{enumerate}
		\item \textbf{Linearity:} 
		
		\begin{eqnarray*}
			\mathcal{D}^{a,b,c}_{d,e,\mu}[\alpha f+\beta g](w) 
			&=& \frac{c_\mu}{( i b)^{\mu+1}} \int_{\mathbb{R}} \Psi^{a,b,c}_{d,e,\mu}(w,v) [\alpha f+\beta g](v)|v|^{2\mu+1} dv \\
			&=& \frac{\alpha c_\mu}{( i b)^{\mu+1}} \int_{\mathbb{R}} \Psi^{a,b,c}_{d,e,\mu}(w,v) f(v)|v|^{2\mu+1} dv \\
			&\quad& + \frac{\beta c_\mu}{( i b)^{\mu+1}} \int_{\mathbb{R}} \Psi^{a,b,c}_{d,e,\mu}(w,v) g(v)|v|^{2\mu+1} dv \\
			&=& \alpha\mathcal{D}^{a,b,c}_{d,e,\mu}[f](w) + \beta\mathcal{D}^{a,b,c}_{d,e,\mu}[g](w).
		\end{eqnarray*}
		
		\item \textbf{Scaling:} We have
		
		\begin{eqnarray*}
			\mathcal{D}^{a,b,c}_{d,e,\mu}[f](k w) 
			&=& \frac{c_\mu}{( i b)^{\mu+1}} \int_{\mathbb{R}} \Psi^{a,b,c}_{d,e,\mu}(k w,v) f(v)|v|^{2\mu+1} dv \\
			&=& \frac{c_\mu}{( i b)^{\mu+1}} \int_{\mathbb{R}} e^{-i(av^2+c(k w)^2+dv+e(k w))} E_{\mu}(-i (k w) / b, v) f(v)|v|^{2\mu+1} dv.
		\end{eqnarray*}
		
		We make the change of variable $v = u/k$, which gives:
		\begin{itemize}
			\item $dv = du/k$
			\item $|v|^{2\mu+1} = |u/k|^{2\mu+1} = |u|^{2\mu+1}/k^{2\mu+1}$
			\item $av^2 = a(u/k)^2 = (a/k^2)u^2$
			\item $dv = d(u/k) = (d/k)u$
			\item $f(v) = f(u/k) = f_k(u)$
		\end{itemize}
		
		For the Dunkl kernel, we have:
		
	\begin{eqnarray*}
		E_{\mu}(-i (k w) / b, u/k) &=& j_{\mu}\left(\frac{(k w)(u/k)}{b}\right) + \frac{-i (k w)(u/k)}{2(\mu+1)b} j_{\mu+1}\left(\frac{(k w)(u/k)}{b}\right)
		\\
		&=& j_{\mu}\left(\frac{w u}{b}\right) + \frac{-i w u}{2(\mu+1)b} j_{\mu+1}\left(\frac{w u}{b}\right) = E_{\mu}(-i w / b, u)
	\end{eqnarray*}
		
		Therefore,
		\begin{eqnarray*}
			\mathcal{D}^{a,b,c}_{d,e,\mu}[f](k w) 
			&=& \frac{c_\mu}{( i b)^{\mu+1}} \int_{\mathbb{R}} e^{-i\left(\frac{a}{k^2}u^2 + c k^2 w^2 + \frac{d}{k}u + e k w\right)} E_{\mu}(-i w / b, u) f_k(u) \frac{|u|^{2\mu+1}}{k^{2\mu+1}} \frac{du}{k} \\
			&=& \frac{1}{k^{2\mu+2}} \frac{c_\mu}{( i b)^{\mu+1}} \int_{\mathbb{R}} e^{-i\left(a'u^2 + c' w^2 + d'u + e' w\right)} E_{\mu}(-i w / b, u) f_k(u)|u|^{2\mu+1} du \\
			&=& \frac{1}{k^{2\mu+2}} \mathcal{D}^{a',b,c'}_{d',e',\mu} [f_k](w),
		\end{eqnarray*}
		with $a' = a/k^2$, $c' = c k^2$, $d' = d/k$, $e' = e k$, and $f_k(v) = f(v/k)$.
	\end{enumerate}
\end{proof}

\begin{theorem}{(The reversibility property)}\label{invthmQF.DT}  
	Let \( a, b, c, d, e \in \mathbb{R}\) such that $b\neq 0$, and $\mu \geq -1/2$. Let $f$ be a function belonging to $L^{1}_{\mu}(\mathbb{R})$, and $\mathcal{D}^{a,b,c}_{d,e,\mu}[f]$ belonging to $L^{1}_{\mu}(\mathbb{R})$. Then the inverse of the quadratic-phase Dunkl transform is given by
	\begin{eqnarray}\label{InvQF.DT} 
		f(v) &=& \left( \mathcal{D}^{-c,-b,-a}_{-e,-d,\mu} \left( \mathcal{D}^{a,b,c}_{d,e,\mu}[f] \right)(w) \right)(v) \nonumber
		\\
		&=& \frac{c_\mu}{(-i b)^{\mu+1}} \int_{\mathbb{R}} \Psi^{-c,-b,-a}_{-e,-d,\mu}(v,w) \mathcal{D}^{a,b,c}_{d,e,\mu}[f](w) |w|^{2\mu+1} dw.
	\end{eqnarray}
\end{theorem}

\begin{proof}
	Let $f$ be a function belongs to $ L^{1}_{\mu}(\mathbb{R})$ such that $\mathcal{D}^{a,b,c}_{d,e,\mu}[f] \in L^{1}_{\mu}(\mathbb{R})$. Then we have
	
	\begin{eqnarray*}
		\left( \mathcal{D}^{-c,-b,-a}_{-e,-d,\mu} \left( \mathcal{D}^{a,b,c}_{d,e,\mu}[f] \right)(w) \right)(v) 
		&=& \frac{c_\mu}{(-i b)^{\mu+1}} \int_{\mathbb{R}} \Psi^{-c,-b,-a}_{-e,-d,\mu}(v,w) \mathcal{D}^{a,b,c}_{d,e,\mu}[f](w) |w|^{2\mu+1} dw \\
		&=& \frac{c_\mu}{(-i b)^{\mu+1}} \int_{\mathbb{R}} \Psi^{-c,-b,-a}_{-e,-d,\mu}(v,w) \\
		&& \times \left[ \frac{c_\mu}{(i b)^{\mu+1}} \int_{\mathbb{R}} \Psi^{a,b,c}_{d,e,\mu}(w,u) f(u) |u|^{2\mu+1} du \right] |w|^{2\mu+1} dw.
	\end{eqnarray*}
	By Fubini's theorem, we can interchange the order of integration, we get
	
	\begin{eqnarray*}
			\left( \mathcal{D}^{-c,-b,-a}_{-e,-d,\mu} \left( \mathcal{D}^{a,b,c}_{d,e,\mu}[f] \right)(w) \right)(v) &=& \frac{c_\mu^2}{(-i b)^{\mu+1}(i b)^{\mu+1}} \int_{\mathbb{R}} \int_{\mathbb{R}} \Psi^{-c,-b,-a}_{-e,-d,\mu}(v,w) \Psi^{a,b,c}_{d,e,\mu}(w,u)
			\\
			\\
			&& \times f(u) |u|^{2\mu+1} |w|^{2\mu+1} du dw.
	\end{eqnarray*}
   By  computing the product of the kernels, we obtain
   
	\begin{eqnarray*}
		\Psi^{-c,-b,-a}_{-e,-d,\mu}(v,w) \Psi^{a,b,c}_{d,e,\mu}(w,u) 
		&=& e^{i(cw^2 + av^2 + ew + dv)} E_{\mu}(i w / b, v) 
		\\ 
		&& \times e^{-i(au^2 + cw^2 + du + ew)} E_{\mu}(-i w / b, u) \\
		&=& e^{i(av^2 + dv)} e^{-i(au^2 + du)} E_{\mu}(i w / b, v) E_{\mu}(-i w / b, u).
	\end{eqnarray*}
	Therefore, we have
	
	\begin{eqnarray*}
		\left( \mathcal{D}^{-c,-b,-a}_{-e,-d,\mu} \left( \mathcal{D}^{a,b,c}_{d,e,\mu}[f] \right)(w) \right)(v) 
		&=& \frac{c_\mu^2}{(-i b)^{\mu+1}(i b)^{\mu+1}} e^{i(av^2 + dv)} \int_{\mathbb{R}} e^{-i(au^2 + du)} f(u) |u|^{2\mu+1} \\
		&& \times \left[ \int_{\mathbb{R}} E_{\mu}(i w / b, v) E_{\mu}(-i w / b, u) |w|^{2\mu+1} dw \right] du.
	\end{eqnarray*}
	Making the change of variable $z = w b^{-1}$, we obtain
	
	\begin{eqnarray*}
		\left( \mathcal{D}^{-c,-b,-a}_{-e,-d,\mu} \left( \mathcal{D}^{a,b,c}_{d,e,\mu}[f] \right)(w) \right)(v) 
		&=& c_\mu^2 e^{i(av^2 + dv)} \int_{\mathbb{R}} e^{-i(au^2 + du)} f(u) |u|^{2\mu+1} \\
		&& \times \left[ \int_{\mathbb{R}} E_{\mu}(i z, v) E_{\mu}(-i z, u) |z|^{2\mu+1} dz \right] du
		\\
		&=& c_\mu e^{i(av^2 + dv)} \int_{\mathbb{R}} E_{\mu}(i z, v)  |z|^{2\mu+1} \\
		&& \times \left[ c_\mu\int_{\mathbb{R}} e^{-i(au^2 + du)} f(u) E_{\mu}(-i z, u) |u|^{2\mu+1} du \right] dz \\
		&=& c_\mu e^{i(av^2 + dv)} \int_{\mathbb{R}} E_{\mu}(i z, v) \mathcal{D}_\mu{g}(z) |z|^{2\mu+1} dz,
	\end{eqnarray*}
	where $g(u)=e^{-i\left(a u^2+d u\right)} f(u)$.
	
	Therefore, we get
	
	\begin{eqnarray*}
		\left( \mathcal{D}^{-c,-b,-a}_{-e,-d,\mu} \left( \mathcal{D}^{a,b,c}_{d,e,\mu}[f] \right)(w) \right)(v) 
		&=&  e^{i(av^2 + dv)} \left[  \mathcal{D}_\mu^{-1} \left(\mathcal{D}_\mu{g}\right) (v)\right] \\
		&=&  e^{i(av^2 + dv)} \left[  \mathcal{D}_\mu \left(\mathcal{D}_\mu{g}\right) (-v)\right] \\
		&=&  e^{i(av^2 + dv)} \left[  e^{-i(av^2 + dv)} f(v)\right]\\
		&=& f(v).
	\end{eqnarray*}
This confirms that the quadratic-phase Dunkl transform is invertible when, with the inverse given by
	\[
	\mathcal{D}^{-c,-b,-a}_{-e,-d,\mu} \circ \mathcal{D}^{a,b,c}_{d,e,\mu} = \text{id}.
	\]
\end{proof}

Now we state and prove the Parseval formula for the quadratic phase Dunkl transform.
\begin{theorem}[Parseval’s formula]
	Let \( a, b, c, d, e \in \mathbb{R}\),  $\mu \geq -1/2$ and  $f$ and $g$ be two functions belonging to $L^{2}_{\mu}(\mathbb{R})$, then we have

		\begin{equation}\label{PFQFDT}
			\int_{\mathbb{R}} f(v) \overline{g(v)} |v|^{2\mu+1} dv = \int_{\mathbb{R}} \mathcal{D}^{a,b,c}_{d,e,\mu}[f](w) \overline{\mathcal{D}^{a,b,c}_{d,e,\mu}[g](w)} |w|^{2\mu+1} d w.
		\end{equation}
\end{theorem}

\begin{proof}
 Let  $f$ and $g$ be two functions in $L^{2}_{\mu}(\mathbb{R})$. According to the relation (\ref{relDTvsQPDT}) between the Dunkl transform and the quadratic phase Dunkl transform, we have 
 
 \begin{equation*}
 	\mathcal{D}^{a,b,c}_{d,e,\mu}[f](w) = \frac{1}{( i b)^{\mu+1}} e^{-i(cw^2+ew)} \mathcal{D}_\mu(f_1)(w / b),
 \end{equation*}
 with $f_1(v)=e^{-i(av^2+dv)}f(v),$  and 
 
  \begin{equation*}
 	\overline{\mathcal{D}^{a,b,c}_{d,e,\mu}[g](w)} = \frac{1}{( -i b)^{\mu+1}} e^{i(cw^2+ew)} \overline{\mathcal{D}_\mu(g_1)(w / b)},
 \end{equation*}
 with $g_1(v)=e^{-i(av^2+dv)}g(v),$

	\begin{eqnarray*}
     \int_{\mathbb{R}} \mathcal{D}^{a,b,c}_{d,e,\mu}[f](w) \overline{\mathcal{D}^{a,b,c}_{d,e,\mu}[g](w)} |w|^{2\mu+1} d w
     &=& \frac{1}{( i b)^{\mu+1}}\frac{1}{( -i b)^{\mu+1}}\int_{\mathbb{R}} \mathcal{D}_\mu(f_1)(w / b) 
     \\ \\
     && \times \overline{\mathcal{D}_\mu(g_1)(w / b)} |w|^{2\mu+1} d w
     \\
      &=& \frac{1}{|b|^{2\mu+2}} \int_{\mathbb{R}} \mathcal{D}_\mu(f_1)(w / b) \overline{\mathcal{D}_\mu(g_1)(w / b)} |w|^{2\mu+1} d w.
     \end{eqnarray*}
	Making the change of variable $(z = w b^{-1})$, we obtain

 	\begin{eqnarray*}
 	\int_{\mathbb{R}} \mathcal{D}^{a,b,c}_{d,e,\mu}[f](w) \overline{\mathcal{D}^{a,b,c}_{d,e,\mu}[g](w)} |w|^{2\mu+1} d w
 	&=&  \int_{\mathbb{R}} \mathcal{D}_\mu(f_1)(z) \overline{\mathcal{D}_\mu(g_1)(z)} |z|^{2\mu+1} d z.
 \end{eqnarray*}
 According the Plancherel theorem for the Dunkl transform, we get
 
  	\begin{eqnarray*}
 	\int_{\mathbb{R}} \mathcal{D}^{a,b,c}_{d,e,\mu}[f](w) \overline{\mathcal{D}^{a,b,c}_{d,e,\mu}[g](w)} |w|^{2\mu+1} d w
 	&=&  \int_{\mathbb{R}} f_1(z) \overline{g_1(z)} |z|^{2\mu+1} d z
 	\\
 		&=&  \int_{\mathbb{R}} e^{-i(az^2+dz)}f(z) \overline{e^{-i(az^2+dz)}g(z)} |z|^{2\mu+1} d z
  \\
    	&=&  \int_{\mathbb{R}} f(z) \overline{g(z)} |z|^{2\mu+1} d z		
 \end{eqnarray*}
This establishes the desired identity, which achieves the proof. 
\end{proof}
 
 \begin{corollary}\label{PlanchQPDT_S}
	Let \( a, b, c, d, e \in \mathbb{R}\) such that $b\neq 0$,  $\mu \geq -1/2$ and $f$ be a function belonging to $\mathcal{S} (\mathbb{R})$, then we have
	
	\begin{equation*}
         \|\mathcal{D}^{a,b,c}_{d,e,\mu}[f]\|_{\mu,2}=  \|f\|_{\mu,2}.
	\end{equation*}
 \end{corollary}
 
\begin{theorem}\label{PlanchQPDT}
	Let \( a, b, c, d, e \in \mathbb{R}\) such that $b\neq 0$,  $\mu \geq -1/2$.
 If $f \in L^{1}_{\mu}(\mathbb{R}) \cap L^{2}_{\mu}(\mathbb{R})$, then its quadratic-phase Dunkl transform $\mathcal{D}^{a,b,c}_{d,e,\mu}[f]$ belongs to $L^{2}_{\mu}(\mathbb{R})$, and we have
       	\begin{equation}\label{PlanchQPDT_L2}
       	\|\mathcal{D}^{a,b,c}_{d,e,\mu}[f]\|_{\mu,2}=  \|f\|_{\mu,2}.
       \end{equation}
\end{theorem}

\begin{proof} 
		Let \( a, b, c, d, e \in \mathbb{R}\) such that $b\neq 0$,  $\mu \geq -1/2$.
 According to Corollary \ref{PlanchQPDT_S} and the density of $\mathcal{S} (\mathbb{R})$ in $L^{2}_{\mu}(\mathbb{R})$, we obtain a unique continuous operator $\mathbb{D}^{a,b,c}_{d,e,\mu}[f]$ on $L^{2}_{\mu}(\mathbb{R})$, that agrees with $\mathcal{D}^{a,b,c}_{d,e,\mu}[f]$ on  $\mathcal{S} (\mathbb{R})$. Let $f$ and $g$ be two function in $\mathcal{S} (\mathbb{R})$, then we have 
	
	  	\begin{eqnarray*}
		\int_{\mathbb{R}} \mathbb{D}^{a,b,c}_{d,e,\mu}[f](w) \overline{\mathcal{D}^{a,b,c}_{d,e,\mu}[g](w)} |w|^{2\mu+1} d w
		&=& \int_{\mathbb{R}} \mathcal{D}^{a,b,c}_{d,e,\mu}[f](w) \overline{\mathcal{D}^{a,b,c}_{d,e,\mu}[g](w)} |w|^{2\mu+1} d w  \\
		&=& \int_{\mathbb{R}} \mathcal{D}^{a,b,c}_{d,e,\mu}[f](w) \overline{\mathbb{D}^{a,b,c}_{d,e,\mu}[g](w)} |w|^{2\mu+1} d w.
	    \end{eqnarray*}
	Hence, by  density of $\mathcal{S} (\mathbb{R})$ in $L^{2}_{\mu}(\mathbb{R})$, we have that for all $f$ and $g$ in $L^{2}_{\mu}(\mathbb{R})$, 
	
		  	\begin{eqnarray*}
		\int_{\mathbb{R}} \mathbb{D}^{a,b,c}_{d,e,\mu}[f](w) \overline{\mathcal{D}^{a,b,c}_{d,e,\mu}[g](w)} |w|^{2\mu+1} d w
		&=& \int_{\mathbb{R}} \mathcal{D}^{a,b,c}_{d,e,\mu}[f](w) \overline{\mathbb{D}^{a,b,c}_{d,e,\mu}[g](w)} |w|^{2\mu+1} d w.
	\end{eqnarray*}
	
  Now, suppose that $f$ is an element of $ L^{1}_{\mu}(\mathbb{R}) \cap L^{2}_{\mu}(\mathbb{R})$ and $g$ belongs to $\mathcal{S} (\mathbb{R})$. Then, we have
  
  \begin{eqnarray*}
  	\int_{\mathbb{R}} \mathcal{D}^{a,b,c}_{d,e,\mu}[f](w) \overline{\mathcal{D}^{a,b,c}_{d,e,\mu}[g](w)} |w|^{2\mu+1} d w
  	&=& \int_{\mathbb{R}} \mathcal{D}^{a,b,c}_{d,e,\mu}[f](w) \overline{\mathbb{D}^{a,b,c}_{d,e,\mu}[g](w)} |w|^{2\mu+1} d w  \\
  	&=& \int_{\mathbb{R}} \mathbb{D}^{a,b,c}_{d,e,\mu}[f](w) \overline{\mathcal{D}^{a,b,c}_{d,e,\mu}[g](w)} |w|^{2\mu+1} d w.
  \end{eqnarray*}
  Therefore, $\mathcal{D}^{a,b,c}_{d,e,\mu}[f]=\mathbb{D}^{a,b,c}_{d,e,\mu}[f]$, a.e., which proves that for all $f \in L^{1}_{\mu}(\mathbb{R}) \cap L^{2}_{\mu}(\mathbb{R})$, its quadratic-phase Dunkl transform $\mathcal{D}^{a,b,c}_{d,e,\mu}[f]$ belongs to $L^{2}_{\mu}(\mathbb{R})$. For the identity (\ref{PlanchQPDT_L2}), follows from  Corollary \ref{PlanchQPDT_S}.
\end{proof}

\section{Quadratic-phase Dunkl  translation} \label{sec4}
In this section we introduce the quadratic-phase Dunkl translation and we develop its fundamental properties. Let us, firstly, recall the definition of the Dunkl translation.

\begin{definition}[Dunkl translation]  The translation operator associated with the Dunkl transform is defined by

\begin{equation}\label{Dtransl}
	\tau_{v,\mu}[f](w)= \int_\mathbb{R} f(\kappa) \mathcal{W}_{\mu}(w,v,\kappa) |\kappa|^{2\mu+1} d\kappa,
\end{equation}
	where  $\mathcal{W}_{\mu}(w,v,\kappa)$ represent the kernel of of Dunkl translation and it  is a nonnegative, symmetric function in three variables, compactly supported in $\kappa$,  given by 

\begin{equation}\label{kernelDtransl}
\mathcal{W}_{\mu}
=\frac{1}{2}\left\{
1-\sigma_{w, v, \kappa}
+\sigma_{\kappa, w, v}
+\sigma_{\kappa, v, w}
\right\}\,
K_{\mu}(|w|, |v|, |\kappa|),
\end{equation}
with

\[
\sigma_{w, v, \kappa}=\frac{w^{2}+v^{2}-\kappa^{2}}{2wv},
\qquad w,v\in\mathbb{R}^{*},
\]
and $K_{\mu}(w, v, \kappa)$ defined for all  \(w,v,\kappa>0\), as below

\[
K_{\mu}(w, v, \kappa)=
\begin{cases}
	\displaystyle
	\frac{2^{\,1-2\mu}\Gamma(\mu+1)}
	{\sqrt{\pi}\,\Gamma(\mu+1/2)}
	 &\frac{\left[((w+v)^{2}-\kappa^{2})(\kappa^{2}-(w-v)^{2})\right]^{\mu-1/2}}
	{(w v \kappa)^{2\mu}}, \\
	& \text{if } |w-v|<\kappa<w+v,
	\\[1em]
	0,
	& \text{otherwise}.
\end{cases}
\]
\end{definition}

The Dunkl translation have the following properties. 

\begin{proposition}
\begin{enumerate}
	\item   For all $f \in L^{p}_{\mu}(\mathbb{R}), 1 \leq p \leq \infty, \mu \geq  -1 / 2$, and for every $w,v \in \mathbb{R}$, we have

$$
\tau_{v,\mu} f(w)=\tau_{w,\mu} f(v),
$$
and

$$
\left\|\tau_{v,\mu} f\right\|_{\mu,p} \leq 4\|f\|_{\mu,p}.
$$
\item  If  $f  \in L^{1}_{\mu}(\mathbb{R})$, and $ v \in \mathbb{R}$, then we have
$$ \mathcal{D}_\mu\left(\tau_{v,\mu} f\right)(w)=E_{\mu}(i w, v)\left(\mathcal{D}_\mu f\right)(w).$$
\end{enumerate}

\end{proposition}

\begin{definition}[Quadratic-phase Dunkl translation]
	Let   \( a, b, c, d, e \in \mathbb{R}\) with $b\neq 0$, and $ \mu > -1/2$.  For $f \in \mathcal{C}(\mathbb{R})$, we define the quadratic-phase Dunkl translation operators  by
	
	\begin{equation}\label{QFDtransl}
		\tau ^{a,b,d}_{w,\mu}[f](v)= \int_\mathbb{R} f(\kappa) \mathcal{W}^{a,b,d}_{\mu}(w,v,\kappa) e^{i(a\kappa^2+d\kappa)}|\kappa|^{2\mu+1} d\kappa,
	\end{equation}
	where  $\mathcal{W}^{a,b,d}_{\mu}(w,v,\kappa)$ is the kernel of the quadratic-phase Dunkl translation operator given 
	as below 
	
	%
	%
	
	\begin{equation*}
		\mathcal{W}^{a,b,d}_{\mu}(w,v,\kappa)= \mathcal{W}_{\mu}(w,v,\kappa)e^{-i\left[ a(w^2+v^2+\kappa^2) +d(w+v+\kappa)\right]},
	\end{equation*}
with  $\mathcal{W}_{\mu}(w,v,\kappa)$, is given by the relation (\ref{kernelDtransl}).
\end{definition}

\begin{properties}
	Let   \( a, b, c, d, e \in \mathbb{R}\) such that $b\neq 0$, and $ \mu > -1/2$. For all $f,g \in \mathcal{C}(\mathbb{R})$ and
	$\alpha,\beta\in\mathbb{R}$, the quadratic-phase Fourier--Bessel translation operator satisfies the following properties.
	\begin{enumerate}
		\item Identity: $  \tau ^{a,b,d}_{0,\mu}=Id$.
		\item Symmetry: $  \tau ^{a,b,d}_{w,\mu}[f](v)=  \tau ^{a,b,d}_{v,\mu}[f](w)$.
		\item Linearity: $  \tau ^{a,b,d}_{w,\mu}[\alpha f+\beta g](v)= \alpha   \tau ^{a,b,d}_{w,\mu}[f](v) + \beta  \tau ^{a,b,d}_{w,\mu}[g](v)$.
	\end{enumerate}
\end{properties}

In the following, we prove the continuity of the quadratic-phase  translation operator from  $L^{p}_{\mu}(\mathbb{R}_+)$  into itself for all $p\in [1,\infty ]$.

\begin{proposition} For every function $h$ in $L^{p}_{\mu}(\mathbb{R})$, $p\in [1,\infty ]$ and $s\in [0,\infty)$, the function $\tau ^{a,b,d}_{w,\mu}[f]$ belongs to $L^{p}_{\mu}(\mathbb{R})$ and we have 
	
	\begin{equation}\label{contQFDTransl}
		\|\tau ^{a,b,d}_{w,\mu}[f]\|_{\mu,p} \leq 4\|f\|_{\mu,p}.
	\end{equation}
\end{proposition}
\begin{proof} For $p=1$ and $\infty$, the proof is trivial. Let us now suppose $p\in (1,\infty)$ and $f$ belongs to $L^{p}_{\mu}(\mathbb{R}_+)$. Then from the definition of the quadratic-phase Dunkl translation operator (\ref{QFDtransl}), we have
	
	\begin{eqnarray*}  \left|  \tau ^{a,b,d}_{w,\mu}[f](v)\right|  &=& \left| \int_\mathbb{R} f(\kappa)  \mathcal{W}^{a,b,d}_{\mu}(w,v,\kappa) e^{au^2+du}|\kappa|^{2\mu+1} d\kappa\right| \\
		&\leq&  \int_\mathbb{R} |f(\kappa)| \left| \mathcal{W}^{a,b,d}_{\mu}(w,v,\kappa)\right| |\kappa|^{2\mu+1} d\kappa,\\
		&=&  \int_\mathbb{R} |f(\kappa)| \left| \mathcal{W}_{\mu}(w,v,\kappa)e^{-i\left[ a(w^2+v^2+\kappa^2) +d(w+v+\kappa)\right]}\right| |\kappa|^{2\mu+1} d\kappa\\
		&=&  \int_\mathbb{R} |f(\kappa)| \left| \mathcal{W}_{\mu}(w,v,\kappa)\right| |\kappa|^{2\mu+1} d\kappa.
	\end{eqnarray*}	
	According to Hölder's inequality, and the fact that 
	\begin{equation*}
		\int_\mathbb{R}  \mathcal{W}_{\mu}(w,v,\kappa)\ |\kappa|^{2\mu+1} d\kappa ; \leq 4,
	\end{equation*} 
	we obtain that
	\begin{equation*}
		\|\tau ^{a,b,d}_{w,\mu}[f]\|_{\mu,p} \leq 4\|f\|_{\mu,p}. 
	\end{equation*}
	
\end{proof}

\section{Quadratic-phase Dunkl convolution}\label{sec5}
In this section we introduce the convolution product associated with the quadratic-phase Dunkl transform and we develop its fundamental properties.

\begin{definition}
	Let   \( a, b, c, d, e \in \mathbb{R}\) with $b\neq 0$, and $ \mu > -1/2$. For \(f, g \in L^1_{\mu}(\mathbb{R})\), we define the  convolution product associated with the quadratic-phase Dunkl transform as below
\end{definition}

\begin{equation}\label{QFDconv}  
	f\star g(w)=  \int_\mathbb{R} \tau ^{a,b,d}_{w,\mu}[f](-v) g(v) e^{i(av^2+dv)}|v|^{2\mu+1} dv,
\end{equation}
with $\tau ^{a,b,d}_{w,\mu}[f]$ is the  translation operators associated with the quadratic-phase Dunkl transform  given by the relation (\ref{QFDtransl}).

The quadratic-phase Dunkl convolution product have the following basic properties.

\begin{properties} Let   \( a, b, c, d, e \in \mathbb{R}\) with $b\neq 0$, $ \mu > -1/2$ and \(f, g, \upsilon \in L^1_{\mu}(\mathbb{R})\), then the quadratic-phase Dunkl convolution product satisfies the following properties:
	\begin{enumerate}
		\item Commutativity: $f\star g= g\star f$.
		\item  Associativity: $(f\star g)\star h =  f\star (g\star h)$.
	\end{enumerate}
\end{properties}

In the following, we introduce the Young's inequality for the quadratic-phase Dunkl convolution product.
\begin{proposition}
	Let   \( a, b, c, d, e \in \mathbb{R}\) such that $b\neq 0$, and $ \mu > -1/2$. Assume that $1 \leq p, q, r \leq \infty$ and $\frac{1}{p}+\frac{1}{q}=\frac{1}{r}+1$. If \(f \in L^p_{\mu}(\mathbb{R})\) and \(g \in L^q_{\mu}(\mathbb{R})\), then $f\star g \in L^r_{\mu}(\mathbb{R})$ and the following inequality holds
	
	$$
	\|f\star g\|_{\mu,r} \leq  4\|f\|_{\mu,p}\|g\|_{\mu,q}.
	$$
	
	In particular, if $f,g \in L^1_{\mu}(\mathbb{R})$, then the  quadratic-phase Dunkl convolution product $h\star g $ is defined almost everywhere on $\mathbb{R}$ and it belongs to $L^1_{\mu}(\mathbb{R})$.
\end{proposition}

\begin{proof}
	Let   \( a, b, c, d, e \in \mathbb{R}\) with $b\neq 0$, and $ \mu > -1/2$. We have
	
	\begin{eqnarray*}
		\left|\tau ^{a,b,d}_{w,\mu}[f](-v) g(v) e^{i(av^2+dv)}\right| 
		= \left(\left|\tau ^{a,b,d}_{w,\mu}[f](-v) \right|^p|g(v)|^q\right)^{\frac{1}{r}}\left(\left|\tau ^{a,b,d}_{w,\mu}[f](-v) \right|^p\right)^{\frac{1}{p}-\frac{1}{r}}\left(|g(v)|^q\right)^{\frac{1}{q}-\frac{1}{q}}.
	\end{eqnarray*}
	Therefore,
	
	\begin{eqnarray*}  
		\int_\mathbb{R}\left|\tau ^{a,b,d}_{w,\mu}[f](-v) e^{i(av^2+dv)} g(v)\right| |v|^{2\mu+1} dv
		& \leq &\left(\int_\mathbb{R}\left|\tau ^{a,b,d}_{w,\mu}[f](-v) \right|^p|g(v)|^q |v|^{2\mu+1} dv\right)^{\frac{1}{r}} \\
		& &\times  \left(\int_\mathbb{R}\left|\tau ^{a,b,d}_{w,\mu}[f](-v) \right|^p |v|^{2\mu+1} dv\right)^{\frac{r-p}{r p}} \\
		& &\times \left(\int_\mathbb{R}|g(v)|^q |v|^{2\mu+1} dv\right)^{\frac{r-q}{r q}}.
	\end{eqnarray*}
	According to the inequality (\ref{contQFDTransl}), we get  
	
	\begin{eqnarray*}
		\left|f\star g(t)\right|^r & \leq & \left\|\tau ^{a,b,d}_{w,\mu}[f]\right\|_{\mu,p}^{r-p}\|g\|_{\mu,q}^{r-q} \int_\mathbb{R}\left|\tau ^{a,b,d}_{w,\mu}[f](-v)\right|^p|g(v)|^q |v|^{2\mu+1} dv \\
		& \leq & 4^{r-p}\|f\|_{\mu,p}^{r-p}\|g\|_{\mu,q}^{r-q} \int_\mathbb{R}\left|\tau ^{a,b,d}_{w,\mu}[f](-v)\right|^p|g(v)|^q |v|^{2\mu+1} dv.
	\end{eqnarray*}
	It follows that,
	
	\begin{eqnarray*}
		\left\|f\star g\right\|_{\mu,r}^r &\leq & 4^{r-p} \|f\|_{\mu,p}^{r-p}\|g\|_{\mu,q}^{r-q} \int_\mathbb{R} \int_\mathbb{R}\left|\tau ^{a,b,d}_{w,\mu}[f](-v)\right|^p|g(v)|^q 
		|v|^{2\mu+1} dv |w|^{2\mu+1} dw \\
		&\leq & 4^{r-p} \|f\|_{\mu,p}^{r-p}\|g\|_{\mu,q}^{r-q} \int_\mathbb{R}|g(v)|^q \int_\mathbb{R}\left|\tau ^{a,b,d}_{v,\mu}[f](-w)\right|^p 
		|w|^{2\mu+1} dw  |v|^{2\mu+1} dv \\
		&\leq & \|f\|_{\mu,p}^{r-p}\|g\|_{\mu,q}^{r-q}\|f\|_{\mu,p}^p\|g\|_{\mu,q}^q \\
		&\leq & \|f\|_{\mu,p}^r\|g\|_{\mu,q}^r.
	\end{eqnarray*}
	Which complete the proof.
\end{proof}

\section{HUP for quadratic-phase Dunkl transform}\label{sec6}
This section is devoted to the study of Heisenberg’s uncertainty principle for the quadratic-phase Dunkl transform. To this end, we begin by stating the Heisenberg’s uncertainty principle for the Dunkl transform introduced by Rosler \cite[Theorem 1.1]{rosler1999uncertainty}.

  \begin{theorem}[HUP for Dunkl transform \cite{rosler1999uncertainty}]    Let $f$ be a function belongs to  $L^2_\mu(\mathbb{R})$. Then we have

\begin{eqnarray}\label{HUPDT}
&& \!\!\!\!\!\!\!\!\!\!\!\!\!\!\!\!\!\!\!\!\!\!\!\!\!\!\!\!\!\!\!\!\!\!\!\!\!\!\!\!	\left(\int_{\mathbb{R}} |v|^2 |g(v)|^2 |v|^{2\mu+1} dv\right) 
\left(\int_{\mathbb{R}} |w|^2 |\mathcal{D}_{\mu}[g](w)|^2 |w|^{2\mu+1} dw\right) \nonumber \\
 &\geq & (\mu + \tfrac{1}{2})^2 \left(\int_{\mathbb{R}} |g(v)|^2 |v|^{2\mu+1} dv\right)^2. 
\end{eqnarray}
  \end{theorem}
  
Now, we introduce the Heisenberg’s uncertainty principle for the quadratic-phase Dunkl transform.

	\begin{theorem}[HUP for quadratic-phase Dunkl transform]
			Let   \( a, b, c, d, e \in \mathbb{R}\) such that $b\neq 0$,  $ \mu > -1/2$ and $f$ be a function belongs to  $L^2_\mu(\mathbb{R})$. Then we have 
\begin{eqnarray}\label{HUPQPDT}
	&& \!\!\!\!\!\!\!\!\!\!\!\!\!\!\!\!\!\!\!\!\!\!\!\!\!\!\!\!\!\!\!\!\!\!\!\!\!\!\!\!
		\left(\int_{\mathbb{R}} |v|^2 |f(v)|^2 |v|^{2\mu+1} dv\right) 
		\left(\int_{\mathbb{R}} |w|^2 |\mathcal{D}^{a,b,c}_{d,e,\mu}[f](w)|^2 |w|^{2\mu+1} dw\right)
		\nonumber\\ 
		&\geq & |b|^{2}  (\mu + \tfrac{1}{2})^2
		\left(\int_{\mathbb{R}} |f(v)|^2 |v|^{2\mu+1} dv\right)^2.
\end{eqnarray}
	\end{theorem}
	
	\begin{proof}
According to the relation (\ref{relDTvsQPDT}) between the Dunkl transform and the quadratic phase Dunkl transform established in Theorem \ref{cont1QPDT}:

\begin{equation}\label{QPDTviaDT}
			\mathcal{D}^{a,b,c}_{d,e,\mu}[f](w) = \frac{1}{( i b)^{\mu+1}} e^{-i(cw^2+ew)} \mathcal{D}_{\mu}[g]\left(\frac{w}{b}\right). 
\end{equation}

where

		\[
g(v) = f(v) e^{-i(av^2+dv)}.
\]		
Therefore, applying Heisenberg inequality for standard Dunkl transform for the function $g$, we obtain
		
\begin{eqnarray}\label{HUPDTg}
	&& \!\!\!\!\!\!\!\!\!\!\!\!\!\!\!\!\!\!\!\!\!\!\!\!\!\!\!\!\!\!\!\!\!\!\!\!\!\!\!\!
		\left(\int_{\mathbb{R}} |v|^2 |g(v)|^2 |v|^{2\mu+1} dv\right) 
		\left(\int_{\mathbb{R}} |w|^2 |\mathcal{D}_{\mu}[g](w)|^2 |w|^{2\mu+1} dw\right) 
		\nonumber\\
		&\geq & (\mu + \tfrac{1}{2})^2 \left(\int_{\mathbb{R}} |g(v)|^2 |v|^{2\mu+1} dv\right)^2. 
\end{eqnarray}
Since $|g(v)| = |f(v)|$, then we have
		\begin{align*}
		I_1=	\int_{\mathbb{R}} |g(v)|^2 |v|^{2\mu+1} dv = \int_{\mathbb{R}} |f(v)|^2 |v|^{2\mu+1} dv. 
		\end{align*}
Hence, we get 

\begin{equation}\label{HPUPDPDI}
I_2=	\int_{\mathbb{R}} |v|^2 |g(v)|^2 |v|^{2\mu+1} dv = \int_{\mathbb{R}} |v|^2 |f(v)|^2 |v|^{2\mu+1} dv.
\end{equation}

Let us denote the second integral in the right hand of the inequality (\ref{HUPDTg}) by

		\[
		I_3 = \int_{\mathbb{R}} |w|^2 |\mathcal{D}_{\mu}[g](w)|^2 |w|^{2\mu+1} dw.
		\]
		and making the substitution $w = u/b$, so we get

		\[
		I_3 = \frac{1}{b^{2\mu+4}}\int_{\mathbb{R}} |u|^2 \left|\mathcal{D}_{\mu}[g]\left(\frac{u}{b}\right)\right|^2 |u|^{2\mu+1}  du.
		\]
Taking absolute values in the equation (\ref{QPDTviaDT}), we get

\[
\left|\mathcal{D}^{a,b,c}_{d,e,\mu}[f](w)\right| = \frac{1}{|b|^{\mu+1}} \left|\mathcal{D}_{\mu}[g]\left(\frac{w}{b}\right)\right|,
\]		
and substituting this into the expression for $I_3$, we get

		\[
I_3 = \frac{1}{b^{2}}\int_{\mathbb{R}} |u|^2 \left|\mathcal{D}^{a,b,c}_{d,e,\mu}[f](w)\right|^2 |u|^{2\mu+1}  du.
\]
Rearranging the expressions of $I_1$, $I_2$ and $I_3$, we get  the desired result
		
\begin{eqnarray*}
	&& \!\!\!\!\!\!\!\!\!\!\!\!\!\!\!\!\!\!\!\!\!\!\!\!\!\!\!\!\!\!\!\!\!\!\!\!\!\!\!\!
	\left(\int_{\mathbb{R}} |v|^2 |f(v)|^2 |v|^{2\mu+1} dv\right) 
	\left(\int_{\mathbb{R}} |w|^2 |\mathcal{D}^{a,b,c}_{d,e,\mu}[f](w)|^2 |w|^{2\mu+1} dw\right)
	\nonumber\\ 
	&\geq & |b|^{2}  (\mu + \tfrac{1}{2})^2
	\left(\int_{\mathbb{R}} |f(v)|^2 |v|^{2\mu+1} dv\right)^2.
\end{eqnarray*}		
		
		This completes the proof.
	\end{proof}

\section*{Conclusion}

This paper introduces the quadratic-phase Dunkl transform, a comprehensive integral transform that unifies many classical and modern transforms—including the Dunkl, fractional Dunkl, linear canonical, and quadratic-phase Fourier transforms—through five real parameters and a multiplicity function.

We established fundamental properties such as boundedness, continuity, and an explicit representation via the classical Dunkl transform. Key results include a Riemann–Lebesgue lemma, an inversion theorem, and Parseval/Plancherel formulas, confirming the quadratic-phase Dunkl transform as a well-defined unitary operator on \(L^2_\mu(\mathbb{R})\).

Furthermore, we introduced associated translation and convolution operators, proving their basic properties and a Young-type inequality. Finally, we derived a Heisenberg-type uncertainty principle for the quadratic-phase Dunkl transform, extending the classical uncertainty principle to this broader setting.

In summary, the quadratic-phase Dunkl transform represents a significant advancement in harmonic analysis, bridging the gap between parameterized phase transforms and Dunkl theory while offering enhanced flexibility for both theoretical and applied mathematics.

\section*{Future works}

Building upon the foundations established in this paper, we are currently investigating further aspects of the quadratic-phase Dunkl transform. Our ongoing research focuses on deriving both qualitative and quantitative uncertainty principles, aiming to refine the Heisenberg-type inequality presented here. In parallel, we are developing the theory of a \textbf{quadratic-phase wavelet Dunkl transform}. For this new transform, we are systematically studying its core properties, including associated uncertainty principles, the behavior of time-frequency localization operators, the construction of wavelet systems, and the analysis of wavelet multipliers. A central objective is to establish a Calderón-type reproducing formula within this extended framework, which would solidify its utility for multiscale analysis and signal processing applications.


\section*{Disclosure statement}
No potential conflict of interest was reported by the author.

\bibliographystyle{abbrv}
\bibliography{bib.QFDT}
		
	\end{document}